\documentclass[11pt, reqno, letterpaper]{amsart}

\usepackage{esint}
\usepackage{amssymb}

\usepackage{bbm}

\usepackage[
hypertexnames=false, colorlinks, 
linkcolor=black,citecolor=black,urlcolor=black, linktocpage=true
]%
{hyperref} 
\hypersetup{bookmarksdepth=3}

	\normalbaselines						  %

\usepackage{graphicx}
\usepackage{
mathrsfs}

\usepackage[T1]{fontenc}

\newcommand{\dd}{\mathrm{d}}

\theoremstyle{definition}

\newtheorem*{df*}{Definition}

\theoremstyle{remark}

\newtheorem*{rem*}{Remark}

\numberwithin{equation}{section}

\newcommand{\ddf}{\buildrel\mathrm{def}\over=}

\newtheorem{theorem}{Theorem}

\newtheorem{corollary}{Corollary}

\newtheorem{remark}{Remark}

\newtheorem{proposition}{Proposition}

\newcommand{\eps}{\varepsilon}

\usepackage{graphicx}
\DeclareGraphicsExtensions{.eps}

\newcommand{\s}{\mathbf}

\begin{document}

\title[The maximum principle]{Isoperimetric functional inequalities via the maximum principle: the exterior differential systems approach}


\author{Paata Ivanisvili}
\thanks{PI is partially supported by the Hausdorff Institute for Mathematics, Bonn, Germany}
\address{Department of Mathematics,  Kent State University, Kent, OH 44240, USA}
\email{ivanishvili.paata@gmail.com}

\author{Alexander Volberg}
\thanks{AV is partially supported by the NSF grants DMS-1265549, DMS-1600065 and by the Hausdorff Institute for Mathematics, Bonn, Germany}
\address{Department of Mathematics,  Michigan State University, East
Lansing, MI 48824, USA}
\email{volberg@math.msu.edu}

\makeatletter
\@namedef{subjclassname@2010}{
  \textup{2010} Mathematics Subject Classification}
\makeatother




%
%


\begin{abstract}
The goal of this note is to give the unified approach to the solutions of a class of isoperimetric problems by relating them to the exterior differential systems studied by R.~Bryant and P.~Griffiths.
\end{abstract}

\date{}
\maketitle

\setcounter{equation}{0}
\setcounter{theorem}{0}

\section{Introduction: a function and its gradient}

In this note we list several classical by now isoperimetric inequalities which can be proved in a unified way. This unified approach reduces them to the so-called exterior differential systems studied by Robert Bryant and Phillip Griffiths. To the best of our knowledge, this is the first article where this connection is made.

Let $d\gamma(x)$ be the standard $n$-dimensional Gaussian measure $d\gamma(x)=\frac{1}{\sqrt{(2\pi)^{n}}} e^{-\frac{|x|^{2}}{2}} dx$.  Set  $\Omega \subset \mathbb{R}$ to be a  bounded closed  interval and let $\mathbb{R}^{+} :=  \{ x \in \mathbb{R} : x \geq 0\}$.  By symbol $C^{\infty}(\mathbb{R}^{n};\; \Omega)$ we denote the smooth,  functions on $\mathbb{R}^{n}$ with values in $\Omega$.   We prove the following theorem:

\begin{theorem}
\label{mth}
If  a real valued function $M(x,y)$ is such that  $M(x,\sqrt{y}) \in C^{2}(\Omega\times \mathbb{R}_{+})$ and it  satisfies the differential inequalities 
\begin{align} 
\label{mdiff}
\begin{pmatrix}
M_{xx}+\frac{M_{y}}{y} & M_{xy}\\
M_{xy} & M_{yy}
\end{pmatrix} \leq 0
\end{align}
then 
\begin{align}\label{gcase}
\int_{\mathbb{R}^{n}} M(f,\| \nabla f\| )d\gamma \leq M\left(\int_{\mathbb{R}^{n}} f d\gamma ,0\right) \quad \text{for all} \quad f\in C^{\infty}(\mathbb{R}^{n};\; \Omega).
\end{align}
\end{theorem}
 
 One can obtain the similar result for uniformly log-concave  probability measures, and the short way to see this  is based on the mass transportation argument.
  In fact, let $d\mu=e^{-U(x)}dx$ be a probability measure such that $U(x)$ is smooth and  $\mathrm{Hess}\, U \geq R \cdot Id$ for some $R>0$. By the result of Caffarelli (see \cite{caf}) there exists  a Brenier map $T = \nabla \phi$ for some convex function $\phi$ such that $T$ pushes forward  $d\gamma$ onto $d\mu$, moreover 
$0 \leq \mathrm{Hess}\, \phi  \leq \frac{1}{\sqrt{R}}\cdot Id$. We apply  (\ref{mdiff}) to $f(x)=g(\nabla \phi (x))$ and use the fact $M_{y} \leq 0$ which follows from (\ref{mdiff}). Since $\| \nabla f(x) \| = \| \mathrm{Hess}\, \phi(x) \, \nabla g(\nabla \phi ) \| \leq \frac{1}{\sqrt{R}} \| \nabla g(\nabla \phi) \|$ we obtain:

\begin{corollary}
\label{sled}
If $M(x,y)$ satisfies $M(x,\sqrt{y})\in C^{2}(\Omega\times \mathbb{R}_{+})$ and \textup{(\ref{mdiff})} then for any  $g \in C^{\infty}(\mathbb{R}^{n};\; \Omega)$ we have 
\begin{align}
\label{corlog}
\int_{\mathbb{R}^{n}} M\left(g, \frac{\|\nabla g\|}{\sqrt{R}}\right) d\mu \leq M\left( \int_{\mathbb{R}^{n}} g d\mu, 0\right), 
\end{align}
where $d\mu = e^{-U(x)}dx$ is a probability measure such that $\mathrm{Hess}\,  U(x) \geq R \cdot Id$. 
\end{corollary}

\bigskip 

In Section~\ref{appl} we present applications of the functional inequality (\ref{corlog}). In Section~\ref{several} we  prove a theorem about  equivalence of some functional inequalities and partial differential inequalities. Corollary~\ref{sled}  is just a consequence of this result. We will notice that our proof of Corollary~\ref{sled}  for general log-concave measure will not differ  from the case of Gaussian measures and it will be completely self-contained (it will not need  the mass transportation argument). 

 In Section~\ref{ets} we describe solutions of (\ref{mdiff}) (in  the important case for us  when the determinant of the matrix in (\ref{mdiff}) is zero) by reducing it to the exterior differential system (EDS) studied by R.~Bryant and P.~Griffiths.  
 
 This allows us to linearize the underlying non-linear PDE that appeared by the requirement of  determinant of the matrix in (\ref{mdiff}) to vanish. In Section~\ref{oned} we investigate  one dimensional case of the results obtained in Section~\ref{several}, and in Section~\ref{further} we present further applications. In particular, we sharpen Beckner--Sobolev inequality (already sharp of course),
 and we show other examples of new isoperimetric inequalities, which one obtains through EDS method.
 
\subsection*{Acknowledgement} We are very grateful to Robert Bryant from whom we learned how to solve an important for our goals non-linear PDE (see~\cite{rob3}).    In Section~\ref{goback} this allows us to explain how one could find the right functions $M(x,y)$ for all the applications mentioned in Section~\ref{appl}, and how to find new functions $M$ each responsible for a particular isoperimetric inequality.

\subsection{A unified approach to classical inequalities via one and the same PDE}
\label{appl}

In this section we list classical isoperimetric inequalities that can be obtained by choosing different solution of the one and the same PDE
\begin{equation}
\label{MA1}
y(M_{xx}M_{yy}- M_{xy}^2)  + M_y M_{yy}=0
\end{equation}
corresponding to different initial values at $y=0$. In the next sections we will show how exterior differential systems (EDS) method allows us to reduce it to a linear PDE, and thus match this classical isoperimetric inequalities with interesting solution of a linear PDE that happened to be just a reverse heat equation.

Then later, starting with subsection \ref{32}, we show that one can choose other interesting solutions of \eqref{MA1}, and, in its turn, this translates to new isoperimetric inequalities. In particular, we will show an instance when Beckner--Sobolev inequality can be further sharpened in an ultimate way.

\subsubsection{Log-Sobolev inequalities: entropy estimates}
Log-Sobolev inequality of Gross (see~\cite{Gross}) states that 
\begin{align}\label{gr}
\int_{\mathbb{R}^{n}} |f|^{2} \ln |f|^{2} d\gamma -\left( \int_{\mathbb{R}^{n}} |f|^{2} d\gamma \right) \ln \left( \int_{\mathbb{R}^{n}} |f|^{2} d\gamma\right) \leq 2 \int_{\mathbb{R}^{n}} \| \nabla f \|^{2}d\gamma 
\end{align}
whenever the right hand side of (\ref{gr}) is well-defined and finite for complex-valued $f$. This implies that if $f$ and $\|\nabla f\|$ are in $L^{2}(d\gamma)$ then $f$ is in the Orlicz space $L^{2}\ln L$.  A proof of Gross uses {\em two-point inequality}  which by central limit theorem establishes  hypercontractivity of the Ornstein--Uhlenbeck semigroup $\| e^{t(\Delta - x\cdot \nabla)}\|_{L^{p}(d\gamma)\to L^{q}(d\gamma)}\leq 1$ for all $t\geq 0$ such that $e^{-2t}\leq \frac{p-1}{q-1}$. Then as a corollary differentiating this estimate at point $t=0$ for $q=2$ one obtains (\ref{gr}).  Earlier than Gross similar {\em two-point inequality} was proved by Aline Bonami (see~\cite{abo1,abo2}). For more on  {\em two-point inequalities}  we refer the reader to \cite{RO}.  For the simple proof of hypercontractivity of Ornstein--Uhlenbeck semigroup we refer the reader to \cite{pivo, ledoux}, and also to earlier works \cite{Hu, Mone}.  Bakry and Emery \cite{BE} extended the inequality for log-concave measures. Namely the inequality 
\begin{align}\label{logsob}
\int_{\mathbb{R}^{n}} f^{2} \ln f^{2} d\mu -\left(\int_{\mathbb{R}^{n}} f^{2} d\mu \right) \ln \left( \int_{\mathbb{R}^{n}} f^{2} d\mu \right) \leq \frac{2}{R}\int_{\mathbb{R}^{n}} \| \nabla f\|^{2} d\mu
\end{align}
holds for a bounded real-valued $f\in C^{1}$ and a log-concave probability measure $d\mu=e^{-U(x)}dx$ such that $\mathrm{Hess}\, U(x) \geq R\cdot Id$.  For further remarks we refer the reader to \cite{BGL}. 

\bigskip

\emph{Proof of \textup{(\ref{logsob})}}: Take 
\begin{align}\label{f1}
M(x,y)=x \ln x - \frac{y^{2}}{2 x}, \quad x >0 \quad \text{and} \quad y\geq 0. 
\end{align}

We have 
\begin{align}\label{logM}
\begin{pmatrix}
M_{xx}+\frac{M_{y}}{y} & M_{xy}\\
M_{xy} & M_{yy}
\end{pmatrix}
=\begin{pmatrix}
-\frac{y^{2}}{x^{3}} & \frac{y}{x^{2}}\\
\frac{y}{x^{2}} & -\frac{1}{x}
\end{pmatrix}\leq 0.
\end{align}
By Corollary~\ref{corlog} we obtain
\begin{align}\label{Blogsob}
\int_{\mathbb{R}^{n}}\left(  g \ln g - \frac{1}{2R} \frac{\| \nabla g\|^{2}}{g}\right) \,  d\mu \leq \left(\int_{\mathbb{R}^{n}} g \,d\mu \right) \ln \left( \int_{\mathbb{R}^{n}} g \,d\mu\right).
\end{align}
 Taking $g=f^{2}$ for positive $f$  and   rearranging terms in (\ref{Blogsob}) we arrive at   (\ref{logsob}).

$\hfill \square$ 

\begin{remark}
The proof we just presented has an obstacle:  
$M(x,\sqrt{y})\notin C^{2}(\mathbb{R}_{+}\times \mathbb{R}_{+})$. In order to avoid this obstacle one has to consider $M^{\eps}(x,y):=M(x+\eps,y)$ for some $\eps >0$. Then surely $M^{\eps}(x,y)$ will satisfies (\ref{mdiff}), what is more  $M^{\eps}(x,\sqrt{y})\in  C^{2}(\mathbb{R}_{+}\times \mathbb{R}_{+})$ and we can repeat the same proof as above for $M^{\eps}(x,y)$. Finally, we just send $\eps \to 0$ assuming that $\int f^{2} d\mu \neq 0$ and we obtain the desired estimate. We should use the same idea in the applications presented below. 
\end{remark}
\bigskip

\bigskip

\subsubsection{Bobkov's inequality: Gaussian isoperimetry}
In \cite{bob3} Bobkov obtained the following functional version of Gaussian isoperimetry. Let $\Phi(x) = \frac{1}{\sqrt{2 \pi }}\int_{-\infty}^{x} e^{-x^{2}/2}dx$, and let $\Phi'(x)$ be a derivative of $\Phi$. Set $I(x) := \Phi'(\Phi^{-1}(x))$.  Then for any locally Lipschitz function $f : \mathbb{R}^{n} \to [0,1],$ we have 
\begin{align}\label{GI}
I\left(\int_{\mathbb{R}^{n}} f d\mu \right) \leq \int_{\mathbb{R}^{n}} \sqrt{I^{2}(f)+\frac{\|\nabla f\|^{2}}{R}} d\mu 
\end{align}
where $d\mu=e^{-U(x)}dx$ is a log-concave probability measure such that $\mathrm{Hess}\, U \geq R \cdot Id$.   Bobkov's proof uses   a {\em two-point inequality:} for all $0 \leq a, b \leq 1$ we have 
\begin{align}\label{twopp}
I\left(\frac{a+b}{2} \right) \leq\frac{1}{2}  \sqrt{I^{2}(a)+\left| \frac{a-b}{2}\right|^{2}}+  \frac{1}{2}  \sqrt{I^{2}(b)+\left| \frac{a-b}{2}\right|^{2}}.
\end{align}
Iterating (\ref{twopp}) appropriately and using central limit theorem Bobkov obtained (\ref{GI}) for the Gaussian measures. By the mass transportation argument one immediately  obtains (\ref{GI}) for uniformly log-concave measures. Notice that $I(0)=I(1)=0$.  Testing (\ref{GI}) for $d\mu =d\gamma$ and $f(x) = \mathbbm{1}_{A}$   where $A$ is a Borel subset of $\mathbb{R}^{n}$ one obtains  Gaussian isoperimetry:  for any Borel measurable set $A \subset \mathbb{R}^{n}$ 
\begin{align}\label{gii}
\gamma^{+}(A) \geq \Phi'(\Phi^{-1}(\gamma(A)))\quad \text{where} \quad \gamma^{+}(A) := \liminf_{\varepsilon \to 0} \frac{\gamma(A_{\varepsilon})-\gamma(A)}{\varepsilon}
\end{align}
 denotes  Gaussian perimeter of $A$, here $A_{\eps}=\{x \in \mathbb{R}^{n} : \mathrm{dist}_{\mathbb{R}^{n}}(A,x)< \varepsilon \}$. For further remarks on (\ref{GI}) see \cite{Ba-L}. Gaussian isoperimetry (\ref{gii}) can be derived also from Ehrhad's inequality (see for example \cite{pivo}). 

\bigskip 

{\em Proof of \textup{(\ref{GI})}}. Take 
\begin{align}\label{f3}
M(x,y)=-\sqrt{I^{2}(x)+y^{2}} \quad \text{where}  \quad x \in [0,1], \quad y \geq 0. 
\end{align}
We have 

\begin{align}\label{smat}
\begin{pmatrix}
M_{xx}+\frac{M_{y}}{y} & M_{xy}\\
M_{xy} & M_{yy}
\end{pmatrix}
=\begin{pmatrix}
-\frac{(I'(x))^{2}y^{2}}{(I^{2}(x)+y^{2})^{3/2}} - \frac{I(x)I''(x)+1}{\sqrt{I^{2}(x)+y^{2}}} & y\frac{I(x)I'(x)}{(I^{2}(x)+y^{2})^{3/2}}\\
y\frac{I(x)I'(x)}{(I^{2}(x)+y^{2})^{3/2}} & -\frac{I^{2}(x)}{(I^{2}(x)+y^{2})^{3/2}}
\end{pmatrix}.
\end{align}
Notice that $I''(x)I(x)+1=0$ therefore (\ref{smat}) is negative semidefinite. So by Corollary~\ref{sled} we obtain 
\begin{align}\label{sheb}
\int_{\mathbb{R}^{n}}-\sqrt{I^{2}(f)+\frac{\| \nabla f\|^{2}}{R}} d\mu \leq -I\left( \int_{\mathbb{R}^{n}} f d\mu \right)
\end{align}
rearranging terms in (\ref{sheb}) we obtain (\ref{GI}) for differentiable $f : \mathbb{R}^{n} \to [0,1]$. Notice that (\ref{sheb}) still holds if $I''(x)I(x)+1\geq 0$ for arbitrary smooth $I(x)$. 

$\hfill \square$

\bigskip



\subsubsection{Poincar\'e inequality and spectral gap.}

Classical Poincar\'e inequality for the Gaussian measure obtained by J.~Nash~\cite{jnash} (see p. 941) states that 
\begin{align}\label{nas}
\int_{\mathbb{R}^{n}} f^{2} d\gamma - \left(\int_{\mathbb{R}^{n}} f d\gamma \right)^{2} \leq \int_{\mathbb{R}^{n}}\| \nabla f \|^{2}d\gamma.
\end{align}
The inequality also says that the spectral gap i.e. the first nontrivial eigenvalue of the self-adjoint positive operator $L=-\Delta + x \cdot \nabla$  in $L^{2}(\mathbb{R}^{n},d\gamma)$ is bounded from  below by $1$.  If $d\mu=e^{-U(x)}dx$ is a probability measure such that   $\mathrm{Hess}\, U \geq R \cdot Id$ then we have 
\begin{align}\label{poin}
\int_{\mathbb{R}^{n}} g^{2} d\mu - \left(\int_{\mathbb{R}^{n}} g d\mu \right)^{2} \leq \frac{1}{R} \int_{\mathbb{R}^{n}}\| \nabla g \|^{2}d\mu.
\end{align}
It is a folklore that  inequality (\ref{poin}), besides of mass transportation argument, follows from the  log-Sobolev inequality (\ref{logsob}):  apply (\ref{logsob}) to the function $f(x)=1+\varepsilon g(x)$ where $\int g d\mu=0$, and send $\varepsilon \to 0$. Then the left hand side of (\ref{logsob}) is $2 \varepsilon^{2} \int g^{2} d\mu + o(\varepsilon^{2})$ whereas  the right hand side of (\ref{logsob}) is $\frac{2\varepsilon^{2}}{R}\int \| \nabla g\|^{2} d\mu$. This gives (\ref{poin}).  
In \cite{BL1} Brascamp and Lieb  obtained the improvement of (\ref{poin}):  instead of $\frac{\| \nabla  g\|^{2}}{R}$ one can put $\langle (\mathrm{Hess} U)^{-1} \; \nabla g, \nabla g \rangle$ in the right hand side of (\ref{poin}), where  we assume that $\mathrm{Hess}\,U$ is just positive. For a simple proof of this improvement  we refer the reader to \cite{CCE} (see also~\cite{BoL} by using Prekopa--Leindler inequality). More subtle result of Bobkov \cite{Bob1} in this direction says that for any log-concave probability measure $d\mu=e^{-U(x)}dx$ one can put $ K \|x-\int x d\mu \|_{L^{2}(d\mu)}^{2} \| \nabla g\|^{2}$ instead of $\frac{\| \nabla g \|^{2}}{R}$ for some universal constant $K>0$. This implies that nonnegative operator $L=-\Delta+\nabla U \cdot x$ has a spectral gap. 

In \cite{beck2} Beckner found an inequality which interpolates in a sharp way between Poincar\'e inequality and log-Sobolev inequality. The inequality was obtained for Gaussian measures but, again,  by mass transportation argument it can be easily translated to a log-concave probability measure. Beckner--Sobolev inequality states  that for $f \in L^{2}(d\mu)$  and $1 \leq p \leq 2$ we have 
\begin{align}\label{bc}
\int |f|^{2} d\mu  - \left( \int |f|^{p}\right)^{2/p} \leq \frac{(2-p)}{R} \int_{\mathbb{R}^{n}} \|\nabla f\|^{2} d\mu
\end{align}
where $d\mu=e^{-U(x)}dx$ is a probability measure such that $\mathrm{Hess}\, U\geq R \cdot Id$. Case $p=1$ gives Poincar\'e inequality (\ref{poin}) and case $p\to 2$ after dividing  (\ref{bc}) by $2-p$ gives (\ref{logsob}). Beckner--Sobolev inequality was studied for different measures in \cite{Lat22}.

\bigskip 

\emph{Proof of \textup{(\ref{bc})}}: Take
 \begin{align}\label{f2}
 M(x,y)=x^{\frac{2}{p}}-\frac{2-p}{p^{2}} x^{\frac{2}{p}-2}y^{2} \quad \text{where} \quad x,y\geq 0\quad  1\leq p \leq 2. 
 \end{align}
 Notice that 
\begin{align}\label{bcM}
\begin{pmatrix}
M_{xx}+\frac{M_{y}}{y} & M_{xy}\\
M_{xy} & M_{yy}
\end{pmatrix}
=\begin{pmatrix}
-\frac{2(2-p)(1-p)(2-3p)x^{\frac{2}{p}-4}y^{2}}{p^{4}} & -\frac{4(2-p)(1-p)x^{\frac{2}{p}-3}y}{p^{3}}\\
-\frac{4(2-p)(1-p)x^{\frac{2}{p}-3}y}{p^{3}} & -\frac{4(2-p)x^{\frac{2}{p}-2}}{p^{2}}
\end{pmatrix}\leq 0.
\end{align}
By Corollary~\ref{sled} we have 
\begin{align}\label{beckD}
\int_{\mathbb{R}^{n}} g^{\frac{2}{p}}-\frac{2-p}{p^{2}}g^{\frac{2}{p}-2}\frac{\|\nabla g\|^{2}}{R} \;d\mu \leq \left(\int_{\mathbb{R}^{n}} g d\mu \right)^{\frac{2}{p}}
\end{align} 
for positive (in fact nonnegative) functions $g$. Now set  $g=|f|^{p}$, and notice that $\| \nabla |f| \| \leq \| \nabla f \|$. After rearranging terms in (\ref{beckD}) we obtain (\ref{bc}).

$\hfill \square$

\subsubsection{3/2 function}
\label{32}
Beckner's inequality (\ref{bc}) can be rewritten in an equivalent form 
\begin{align}\label{bec1}
\int_{\mathbb{R}^{n}} f^{p} d\gamma - \left( \int_{\mathbb{R}^{n}} f d\gamma \right)^{p} \leq \frac{p(p-1)}{2}\int_{\mathbb{R}^{n}} f^{p-2} \|\nabla  f\|^{2}d\gamma, \quad p \in [1,2].
\end{align}
In fact inequality (\ref{bec1}) can be essentially improved for $p\in (1,2)$. We will illustrate the improvement in the case $p=3/2$ and for the general case we should refer the reader to our recent paper \cite{PIVO5} which is based on the application of Theorem~\ref{mth}. 

The following inequality valid for all smooth bounded nonnegative $f$ was proved in our recent paper \cite{PIVO5}:
\begin{align}\label{bec2}
&\int_{\mathbb{R}^{n}} f^{3/2} d\gamma - \left( \int_{\mathbb{R}^{n}} f d\gamma \right)^{3/2} \leq \nonumber \\
&\int_{\mathbb{R}^{n}} \left(f^{3/2}-\frac{1}{\sqrt{2}}(2f - \sqrt{f^{2}+\|\nabla f\|^{2}})\sqrt{f+\sqrt{f^{2}+\|\nabla f\|^{2}}} \right)d\gamma.
\end{align}
Inequality (\ref{bec2}) improves Beckner's bound (\ref{bec1}) for $p=3/2$. Indeed, notice that we have the following \emph{pointwise} inequality
\begin{align}\label{imp1}
x^{3/2}-\frac{1}{\sqrt{2}}\left(2x-\sqrt{x^{2}+y^{2}} \right)\sqrt{x+\sqrt{x^{2}+y^{2}}} \leq \frac{3}{8} x^{-1/2}y^{2}, \quad x,y \geq 0,
\end{align}
which follows from the homogeneity, i.e., take $x=1$. 
By plugging $f$ for $x$, $|\nabla f|$ for $y$ and integrating we see that \eqref{bec2} improves on \eqref{bec1}.

Inequality (\ref{imp1}) is always strict except when $y=0$. Also notice that when $y \to +\infty$ the right hand side  of (\ref{imp1}) increases as $y^{2}$ whereas the left hand side of (\ref{imp1}) increases as $y^{3/2}$. It should be mentioned as well that when $x \to 0$ the difference in (\ref{imp1}) goes to infinity. The only place where the quantities in (\ref{imp1}) are comparable is when $y/x \to 0$. We should notice that since the left hand side of (\ref{imp1}) is decreasing function in $x$ (see \cite{PIVO5}), and when $x=0$ it becomes $\frac{y^{3/2}}{\sqrt{2}}$ then it follows from (\ref{bec2}) 
\begin{align}\label{mconc}
\int_{\mathbb{R}^{n}} f^{3/2} d\gamma - \left( \int_{\mathbb{R}^{n}} f d\gamma \right)^{3/2} \leq \frac{1}{\sqrt{2}}\int_{\mathbb{R}^{n}} \| \nabla f \|^{3/2}d\gamma.
\end{align}  
Inequality (\ref{mconc}) gives some information about the measure concentration of $\gamma$. 

\bigskip 

{\em Proof of \textup{(\ref{bec2})}}. Take 
\begin{align}\label{ax1}
M(x,y) = \frac{1}{\sqrt{2}}\left(2x-\sqrt{x^{2}+y^{2}} \right)\sqrt{x+\sqrt{x^{2}+y^{2}}} \quad \text{where} \quad x,y \geq 0.
\end{align}

We have 
\begin{align}\label{imp2}
\begin{pmatrix}
M_{xx}+\frac{M_{y}}{y} & M_{xy}\\
M_{xy} & M_{yy}
\end{pmatrix}
=\frac{3\sqrt{2}}{8\sqrt{x^{2}+y^{2}}}\begin{pmatrix}
-\frac{y^{2}}{(x+\sqrt{x^{2}+y^{2}})^{3/2}} & \frac{y}{\sqrt{x+\sqrt{x^{2}+y^{2}}}}\\
\frac{y}{\sqrt{x+\sqrt{x^{2}+y^{2}}}} & - \sqrt{x+\sqrt{x^{2}+y^{2}}}
\end{pmatrix}.
\end{align}
Clearly (\ref{imp2}) is negative semidefinite. So by Corollary~\ref{sled} we obtain 
\begin{align}\label{in10}
\int_{\mathbb{R}^{n}} \frac{1}{\sqrt{2}}\left(2f-\sqrt{f^{2}+\frac{\| \nabla f\|^{2}}{R}} \right)\sqrt{f+\sqrt{f^{2}+\frac{\|\nabla f\|^{2}}{R}}} d\mu \leq \left(\int_{\mathbb{R}^{n}} f d\mu \right)^{3/2}.
\end{align}
This is of course \eqref{bec2}
for the Gaussian measure $\gamma$:  by taking $R=1$ and rearranging terms in (\ref{in10}) we obtain (\ref{bec2}). 

\bigskip

\subsubsection{Banaszczyk's problem: \textup{(B)} Theorem}
The problem was proposed by W.~Banaszczyk  (see for example~\cite{Lat}) which says that given symmetric convex body $K \subset \mathbb{R}^{n}$ the function $\phi(t) = \gamma(e^{t} K)$ is log-concave on $\mathbb{R}$. The problem was solved in~\cite{CFM}: clearly one only needs to check log-concavity at one point:  $(\ln \phi(t))''|_{t=0}\leq 0$. This is the same as 
\begin{align}\label{muravei}
\int_{\mathbb{R}^{n}}\|x\|^{4} d\gamma_{k} - \left( \int_{\mathbb{R}^{n}} \|x\|^{2} d\gamma_{K}\right)^{2}\leq 2 \int_{\mathbb{R}^{n}}\|x\|^{2} d\gamma_{K}
\end{align}
where 
\begin{align*}
d\gamma_{K} = \frac{\mathbbm{1}_{K}(x)e^{-\|x\|^{2}/2}dx}{\int_{K}e^{-\|y\|^{2}/2}dy}=e^{-\|x\|^{2}/2-\psi(x)}dx
\end{align*}
where a convex function $\psi$ is a constant on $K$ and it is $+\infty$ outside of the set $K$.  In other words one can assume that $d\gamma_{K}=e^{-U(x)}dx$ is a probability measure  where $U(x)$ is even and   such that $\mathrm{Hess}\, U \geq Id$.  Setting $f(x)=\|x\|^{2}$ then inequality (\ref{muravei}) can be rewritten as follows
\begin{align}\label{cfm1}
\int_{\mathbb{R}^{n}} f^{2} d\mu - \left(\int_{\mathbb{R}^{n}} f d\mu \right)^{2} \leq \frac{1}{2}\int_{\mathbb{R}^{n}} \| \nabla f \|^{2} d\mu.
\end{align} 
which is better than Poincar\'e inequality (\ref{poin}). This is a key ingredient in (B) Theorem and  it was proved by Cordero-Erausquin--Fradelizi--Maurey in~\cite{CFM} that (\ref{cfm1}) holds provided that 
$\int_{\mathbb{R}^{n}}\nabla f d\mu=0$, and $d\mu=e^{-U(x)}dx$ is a probability measure such that $\mathrm{Hess} \,U \geq Id$ (which is true for  $f(x)=\|x\|^{2}$). 

If one tries to apply Corollary~\ref{sled} then the right choice of the function $M$ must be 
\begin{align}\label{f5}
M(x,y)=x^{2}-\frac{y^{2}}{2}
\end{align}
 but unfortunately this function does not satisfy (\ref{mdiff}). However,  we want (\ref{cfm1}) to hold only for the functions such that $\int \nabla f d\mu=0$ therefore one can slightly modify the proof of Theorem~\ref{mth}  in order to obtain (\ref{cfm1}). In Section~\ref{oned} we will show how it works and we will present a different proof of (\ref{cfm1}) with the extra conditions that $f$ is even and $d\mu$ is even (which definitely is enough for the (B) Theorem).  

\bigskip

\subsubsection{$\Phi$-entropy}\label{ppent}
Let $\Phi : \Omega \to \mathbb{R}$ be a convex function. Given a probability measure $d\mu$ on $\mathbb{R}^{n}$ define $\Phi$-entropy (see \cite{chafai}) as follows 
\begin{align*}
 \boldsymbol{\mathrm{Ent}}_{\mu}^{\Phi}(f) \stackrel{\mathrm{def}}{=} \int_{\mathbb{R}^{n}}\Phi(f)d\mu - \Phi\left(\int_{\mathbb{R}^{n}} f d\mu \right).
\end{align*}
Corollary~\ref{sled} provides us with  systematic approach to finding the bounds of $\Phi$-entropy for uniformly log-concave measures $d\mu$. Indeed, let us illustrate it on the example of the Gaussian measure. Given a convex function $\Phi$ on $\Omega\subset \mathbb{R}$ let $M(x,y)$ be such that
$M(x,0)=\Phi(x)$, $M(x,\sqrt{y}) \in C^{2}(\Omega\times \mathbb{R}_{+})$ and $M$ satisfies (\ref{mdiff}). Then by Theorem~\ref{mth} we obtain 
\begin{align*}
\int_{\mathbb{R}^{n}} \Phi (f(x)) d\gamma - \Phi \left( \int_{\mathbb{R}^{n}} f d\gamma \right) \leq\int_{\mathbb{R}^{n}} [M(f,0) - M(f, \|\nabla f\|)] d\gamma.
\end{align*}

In our recent paper \cite{PIVO5} we do find the bounds of $\Phi$ entropy as an application of Theorem~\ref{mth} for the following fundamental examples
\begin{align}
  &\Phi(x) = x^{p} \quad \text{for}  \quad p \in \mathbb{R}\setminus [0,1]; \label{91}\\
  &\Phi(x) = -x^{p} \quad \text{for}  \quad p \in (0,1);\label{92}\\
  &\Phi(x) = e^{x};\label{93}\\
  &\Phi(x) = -\ln x.\label{94}
 \end{align}
 
 Finding the best possible $M$ is based on solving a PDE problem (\ref{Monge}) with  boundary conditions (\ref{91}, \ref{92}, \ref{93}, \ref{94}) (see Section~\ref{conclusions}, Section~\ref{ets} and \cite{PIVO5}). 

\bigskip

\subsubsection{Yet another isoperimetric inequality obtained by EDS method}\label{arccos}
In Section~\ref{ets}  we consider a peculiar example  (see Section~\ref{peculiar}) of the elliptic solution of PDE \eqref{MA1}
with initial data
\begin{align*}
M(x,0)= x \arccos(-x)+\sqrt{1-x^{2}}  \quad  \text{for} \quad  x \in [-1,1]
\end{align*}
which is not related to the applications that we have discussed before, but which gives yet another example of a new isoperimetric inequality.  It looks like a useful one in particular because Poincar\'e inequality is its corollary.

\bigskip

\subsubsection{Concluding remarks.}\label{conclusions}
As we shall notice in order to use Theorem~\ref{mth}  for the applications to functional (and thereby isoperimetric) inequalities, there is a difficulty: one  has to find the  right function $M(x,y)$, for example such as  (\ref{f1}), (\ref{f2}), (\ref{f3}), (\ref{ax1}), (\ref{f5}) and functions $M$ mentioned in Section~\ref{ppent} (see~\cite{PIVO5}).

  If one  knows what inequality should be  proved then one can try to guess what function $M(x,y)$ one has to choose: in the integrand one needs to set $g=x$ and $\| \nabla g \|=y$ and then integrand in terms of $x$ and $y$ will be $M(x,y)$. 


\bigskip 

In general  finding $M(x,y)$ will be based purely on solving PDEs. Let us recall the discussions of Section~\ref{ppent}. First notice that given, for example,  a convex function $\Phi: \Omega \to \mathbb{R}$  and suppose one wants to find an optimal  \emph{error term} in the Jensen's inequality ($\Phi$-entropy (see~\cite{chafai})) 
\begin{align*}
0\leq \int_{\mathbb{R}^{n}} \Phi (f(x)) d\gamma - \Phi \left( \int_{\mathbb{R}^{n}} f d\gamma \right) \leq \int_{\mathbb{R}^{n}} \mathrm{Error}(f, \|\nabla f\|) d\gamma \quad \text{for all} \quad   f \in C^{\infty}(\mathbb{R}^{n};\Omega). 
\end{align*}
If we find $M(x,y) \in C^{2}(\Omega\times\mathbb{R}_{+})$ such that $M(x,0)=\Phi(x)$ and $M(x,y)$ satisfies (\ref{mdiff}) then by Theorem~\ref{mth} we can find a possible error term as follows
\begin{align}\label{error}
\int_{\mathbb{R}^{n}} \Phi (f(x)) d\gamma - \Phi \left( \int_{\mathbb{R}^{n}} f d\gamma \right) \leq\int_{\mathbb{R}^{n}} [M(f,0) - M(f, \|\nabla f\|)] d\gamma.
\end{align}
In fact we would like to minimize  the error term which  corresponds to maximize $M(x,y)$  under the constraints (\ref{mdiff}) and $M(x,0)=\Phi(x)$. This suggests that partial differential inequality (\ref{mdiff}) should degenerate. Indeed, if  $\lambda_{1}(x,y)$ and $\lambda_{2}(x,y)$ denote eigenvalues of the matrix in (\ref{mdiff}) then condition (\ref{mdiff}) becomes $\lambda_{1}+\lambda_{2}\leq 0$ and $\lambda_{1} \cdot \lambda_{2} \geq 0$. If we have strict inequality $\lambda_{1}\cdot \lambda_{2} >0$ then $\lambda_{1}+\lambda_{2}<0$. In this case we can slightly perturb $M$ at a point $(x,y)$ so that to make $M(x,y)$ larger but still keep the inequality $\lambda_{1} \cdot \lambda_{2} >0$. Clearly the condition $\lambda_{1}+\lambda_{2}<0$ still holds. We can continuo  perturbing  $M$ until (\ref{mdiff}) degenerates. Therefore we will seek $M(x,y)$ among those functions which in addition with  (\ref{mdiff}) also satisfy a {\em degenerate elliptic  Monge--Amp\'ere equation of general type}:  
\begin{align}\label{Monge}
\det 
\begin{pmatrix}
M_{xx}+\frac{M_{y}}{y} & M_{xy}\\
M_{xy} & M_{yy}
\end{pmatrix} = M_{xx}M_{yy}-M_{xy}^{2}+\frac{M_{y}M_{yy}}{y}=0 
\end{align}
for $(x,y) \in \Omega \times \mathbb{R}_{+}$.

For example in log-Sobolev~(\ref{logsob})  and in Bobkov's inequality~(\ref{GI}) determinant of the matrices   (\ref{logM}) and (\ref{smat}) are zero. In Beckner--Sobolev inequality~(\ref{bc}) determinant of (\ref{bcM}) is zero if and only if $p=1, 2$. Notice that these are exactly cases when Beckner--Sobolev inequality interpolates Poincar\'e  and log-Sovbolev inequality. Moreover, since the determinant in Beckner--Sobolev inequality is not zero for $p\in (1,2)$ this indicates that one should improve the inequality, and this is exactly what was done in (\ref{bec2}).  We refer the reader to our recent paper \cite{PIVO5} where we do improve Beckner--Sobolev inequality by solving elliptic Monge--Amp\`ere equation (\ref{Monge}) with a boundary condition $M(x,0)=x^{p}$ with  $p \in \mathbb{R}$.  

In Section~\ref{ets} we will show that thanks to the exterior differential systems studied by R.~Bryant and P.~Griffiths (see \cite{br1, br2, br3}) nonlinear equation~(\ref{Monge}) can be reduced  (after suitable change of variables) to linear backwards heat equation. 
In Section~\ref{app12} we will illustrate this on the examples 
\begin{align*}
M(x,0)=x \ln x, \; M(x,0)=x^{2}, \; M(x,0)=-I(x) \quad \text{and} \quad  M(x,0)=x^{3/2}
\end{align*} 
which correspond to log-Sobolev, Poincar\'e,  Bobkov and $3/2$  inequalities. 

To justify the claim that Section~\ref{ets} makes approach to bounds of $\Phi$-entropy systematic we do consider a peculiar example  (see Section~\ref{peculiar})
\begin{align*}
M(x,0)= x \arccos(-x)+\sqrt{1-x^{2}}  \quad  \text{for} \quad  x \in [-1,1]
\end{align*}
which is not related to the applications that we have discussed before.

\section{Function of  variables $D^{\alpha}{\bf f}$. The proof of Theorem~\ref{mth}}
\label{several}
We prove here Theorem ~\ref{mth}   and even more general Theorem \ref{absth}. The reader who a priori believes in Theorem ~\ref{mth}  can skip to the next Section ~\ref{ets} devoted to the exterior differential systems (EDS) method of finding the elliptic solutions of PDE \eqref{MA1} (by elliptic solutions we mean the solutions $M$ satisfying the condition \eqref{mdiff} on $M$).

Let $d\mu=e^{-U(x)}dx$ be a log-concave measure such that $U$ is smooth and $\mathrm{Hess}\, U \geq R \cdot Id$. Set $L=\Delta  - \nabla U \cdot \nabla $. Then  $-L$ is a self-adjoint positive operator in $L^{2}(\mathbb{R}^{n},d\mu)$, moreover by~(\ref{poin})  it has a spectral gap.  Let $P_{t}:=e^{tL}$ be the corresponding semigroup generated by $L$. 
 Let $\alpha=(\alpha_{0},\ldots, \alpha_{m})$ where   $\alpha_{j}=(\alpha_{j}^{1},\ldots, \alpha_{j}^{n})$ is a multi index of size $n$ and  $\alpha_{j}^{i}\in \mathbb{N}\cup \{ 0\}$ for each $j=0,\ldots m$ and $i=1,\ldots, n$. Let $|\alpha_{j}|$ be the length of the multi index i.e. $|\alpha_{j}|= \alpha_{j}^{1}+\ldots+\alpha_{j}^{n}$. By $D^{\alpha_{j}}$ we denote the  differential operator 
 \begin{align}\label{dop}
 D^{\alpha_{j}}=\frac{\partial^{|\alpha_{j}|}}{\partial x^{\alpha_{j}^{1}}_{1}\cdots \partial x_{n}^{\alpha_{j}^{n}}}.
 \end{align}
 Further we fix some {\em multi-multi} index  $\alpha = (\alpha_{0},\ldots, \alpha_{m})$  where each $\alpha_{j}$ is a multi index of size $n$ as above.

\subsection*{Test functions $C^{\infty}(\mathbb{R}^{n}; \Lambda)$}
 Let $\Lambda$ be a closed convex subset of $\mathbb{R}^{m}$. By $C^{\infty}(\mathbb{R}^{n}; \Lambda)$ we denote the set of test functions  ${\bf f}=(f_{0},\ldots, f_{m}): \mathbb{R}^{n} \to \Lambda$ i.e., smooth bounded vector functions with values in $\Lambda$.  Let 
 \begin{align*}
 D^{\alpha} {\bf f} = (D^{\alpha_{0}}f_{0},\ldots, D^{\alpha_{m}}f_{m}) \quad \text{and} \quad  P_{t} {\bf f} := (P_{t}f_{0},\ldots, P_{t}f_{m}).
 \end{align*}
 We require that $C^{\infty}(\mathbb{R}^{n}; \Lambda)$ is closed under taking $D^{\alpha}$, i.e., $D^{\alpha}C^{\infty}(\mathbb{R}^{n}; \Lambda) \subset C^{\infty}(\mathbb{R}^{n}; \Lambda)$.  Linearity and positivity of $P_{t}$ implies that $P_{t} {\bf f}, P_{t} D^{\alpha} {\bf f} \in \Lambda$ for any $f \in C^{\infty}(\mathbb{R}^{n}; \Lambda)$. 
 

Let $B(u_{1},\ldots, u_{m}) : \Lambda \to \mathbb{R}$ be a smooth (at least $C^{2}$) function, such that $P_{t} B(D^{\alpha}{\bf f})$ is well defined for all $t\geq 0$.  Set 
\begin{align*}
 [L,D^{\alpha}]{\bf f} \ddf ([L,D^{\alpha_{0}}]f_{0},\dots,[L,D^{\alpha_{m}}]f_{m})  \quad \text{and} \quad \Gamma(D^{\alpha}{\bf f}) \ddf \{ \langle \nabla D^{\alpha_{i}} f_{i}, \nabla D^{\alpha_{j}} f_{j} \rangle \}_{i,j=0}^{m}
\end{align*}
where $\Gamma(D^{\alpha}{\bf f})$ denotes $(m+1)\times (m+1)$, and $[A,B] = AB-BA$ denotes commutator of $A$ and $B$.
\begin{theorem}\label{absth}
The following conditions are equivalent:
\begin{itemize}
\item[\textup{(i)}] $\nabla B(D^{\alpha}{\bf f})\cdot [L,D^{\alpha}]{\bf f} + \mathrm{Tr}[ \mathrm{Hess}\, B(D^{\alpha}{\bf f}) \; \Gamma(D^{\alpha}{\bf f})]\leq 0$ for all $f\in C^{\infty}(\mathbb{R}^{n}; \Lambda)$. 

\item[\textup{(ii)}] $P_{t}[B(D^{\alpha} {\bf f})](x) \leq B(D^{\alpha} [P_{t} {\bf f}] (x))$ for all $t\geq 0$, $x \in \mathbb{R}^{n}$ and  $f\in C^{\infty}(\mathbb{R}^{n}; \Lambda)$. 
\end{itemize}
\end{theorem}
\begin{proof}
$(i)$ implies $(ii)$: let $V(x,t) = P_{t}[B(D^{\alpha} {\bf f})](x) - B(D^{\alpha} [P_{t} {\bf f}] (x))$. Notice that  
\begin{align}\label{maxp}
&(\partial_{t}-L) V(x,t)=(L-\partial_{t}) B(D^{\alpha} [P_{t} {\bf f}] (x))= \nonumber \\
&\sum_{j} \frac{\partial B}{\partial u_{j}} LD^{\alpha_{j}}P_{t}f_{j} + \sum_{i,j} \frac{\partial^{2} B}{\partial u_{i} \partial u_{j}} \nabla D^{\alpha_{i}}P_{t}f_{i} \cdot \nabla D^{\alpha_{j}}P_{t} f_{j} - \sum_{j} \frac{\partial B}{\partial u_{j}} D^{\alpha_{j}}LP_{t} f_{j}=\nonumber \\
& \nabla B(D^{\alpha}{P_{t} \bf f})\cdot [L,D^{\alpha}]{P_{t} \bf f} + \mathrm{Tr}(\mathrm{Hess}\, B(D^{\alpha}{P_{t}\bf f}))\; \Gamma(D^{\alpha}{P_{t}\bf f}))\leq 0 
\end{align}
The last inequality follows from (i) and  the fact that $P_{t} {\bf f}(x) \in \Lambda$. Indeed,  we can find a function ${\bf g} \in C^{\infty}(\mathbb{R}^{n}; \Lambda)$ such that ${\bf g} = P_{t} {\bf f}$ in a neighborhood of $x$ and we can apply (i) to ${\bf g}$. 

By maximum principle we obtain $V(x,t)\leq \sup_{x} V(x,0)=0$. Another way (without maximum principle) is that 
\begin{align}\label{lees}
V(x,t)=\int_{0}^{t} \frac{\partial }{\partial s} P_{s}B(D^{\alpha} P_{t-s} {\bf f}) ds=\int_{0}^{t}P_{s} \left[ \left(L-\frac{\partial}{\partial t} \right) B(D^{\alpha} P_{t-s} {\bf f})\right] ds,
\end{align}
and  the integrand in (\ref{lees}) is non positive by (\ref{maxp}). 

\bigskip 

(ii) impies (i):  for all ${\bf f} \in C^{\infty}(\mathbb{R}^{n}; \Lambda)$ we have 
\begin{align*}
&0 \geq \lim_{t \to 0} \frac{V(x,t)}{t} = \lim_{t\to 0} \frac{V(x,t)-V(x,0)}{t}=\frac{\partial }{\partial t} V(x,t)|_{t=0}=\\
&\nabla B(D^{\alpha}{\bf f})\cdot [L,D^{\alpha}]{\bf f} + \mathrm{Tr}(\mathrm{Hess}\, B(D^{\alpha}{\bf f}))\; \Gamma(D^{\alpha}{\bf f})).
\end{align*}
\end{proof}

\begin{remark}

We notice that  if one considers diffusion semigroups generated by 
\begin{align*}
L=\sum_{ij} a_{ij}(x) \frac{\partial^{2}}{\partial x_{i} \partial x_{j}}+ \sum_{j} b_{j}(x) \frac{\partial }{\partial x_{j}}
\end{align*}
where $A=\{ a_{ij}\}_{i,j=1}^{n}$ is  positive then absolutely nothing changes in Theorem~\ref{absth} except the matrix $\Gamma(D^{\alpha}{\bf f})$ takes the form 
 \begin{align*}
 \Gamma(D^{\alpha}{\bf f})=\{ \nabla D^{\alpha_{i}}f_{i} \, A \, (\nabla D^{\alpha_{j}}f_{j})^{T}\}_{i,j=0}^{m}.
 \end{align*}
 \end{remark}

\medskip

\subsection{Proof of Theorem~\ref{mth}}
Consider a special case when $n=m$,  $ {\bf f} =\underbrace{(f,\ldots, f)}_{n+1}$, $\alpha_{0}=\underbrace{(0,\ldots, 0)}_{n}$, $\alpha_{1}=(1,0,\ldots, 0),\ldots$, and $\alpha_{n}=(0,\ldots,0,1)$. Then $D^{\alpha} {\bf f} = (f, \nabla f)$, and given that $L=\Delta-\nabla U \cdot \nabla$ we obtain 
\begin{align*}
\nabla B(D^{\alpha} {\bf f}) \cdot [L,D^{\alpha}]{\bf f} = \nabla_{1,\ldots, n} B \, (\mathrm{Hess}\, U) (\nabla f)^{T}.
\end{align*}
Here $\nabla_{1,\ldots, n} B$ is a gradient of $B(u_{0},\ldots, u_{n})$ taken with respect to $u_{1},\ldots, u_{n}$ variables. 
Assume that $f$ takes values in the closed convex set $\Omega \subset \mathbb{R}$.
Take 
\begin{align}\label{bfun1}
B(u_{0},\ldots, u_{n})=M\left(u_{0}, \sqrt{\frac{u_{1}^{2}+\cdots +u_{n}^{2}}{R}} \right),
\end{align}
where $M(x,\sqrt{y}) \in C^{2}(\Omega \times \mathbb{R}_{+})$ satisfies (\ref{mdiff}).  Notice that $M_{y} \leq 0$. Indeed, if we multiply the first diagonal entry of (\ref{mdiff}) by $y$ and send $y \to 0$ we obtain $M_{y}(x,0)\leq 0$. On the other hand  since the second diagonal entry of (\ref{mdiff}) is nonpositive we obtain $M_{y}(x,y)\leq 0$ for all $y$.  

Next we notice
\begin{align*}
\nabla_{1,\ldots, n} B (D^{\alpha}{\bf f}) = \frac{M_{y}}{\|\nabla f\| \sqrt{R}}\nabla f.
\end{align*}
Since $M_{y}\leq 0$ and $\mathrm{Hess}\, U \geq R \cdot Id$, we have  
\begin{align*}
\nabla B(D^{\alpha} {\bf f}) \cdot [L,D^{\alpha}]{\bf f} = \frac{M_{y}}{\|\nabla f\| \sqrt{R}}\nabla f (\mathrm{Hess}\, U) (\nabla f)^{T} \leq \sqrt{R} \|\nabla f\| M_{y}.
\end{align*}
Therefore 
\begin{align*}
\nabla B(D^{\alpha}{\bf f})\cdot [L,D^{\alpha}]{\bf f} + \mathrm{Tr}(\mathrm{Hess}\, B(D^{\alpha}{\bf f}))\; \Gamma(D^{\alpha}{\bf f}))\leq \mathrm{Tr} (W\, \Gamma(D^{\alpha}{\bf f}))
\end{align*}
where 
\begin{align*}
W = 
\begin{bmatrix}
    \partial^{2}_{00}B+\frac{\sqrt{R}\cdot M_{y}}{\|\nabla f\|}       & \partial^{2}_{01}B  & \dots & \partial^{2}_{0n}B  \\
    \partial^{2}_{10}B      & \partial_{11}^{2}B & \dots & \partial^{2}_{1n}B \\
    \hdotsfor{4} \\
    \partial^{2}_{n0}B    &\partial^{2}_{n1}B & \dots & \partial^{2}_{nn}B 
\end{bmatrix}
\end{align*}
where $\partial^{2}_{ij}B = \frac{\partial^{2} B}{\partial u_{i} \partial u_{j}} $.  We will show that $W \leq 0$, and then we will obtain
$\mathrm{Tr} (W\, \Gamma(D^{\alpha}{\bf f})) \leq 0$ because $ \Gamma(D^{\alpha}{\bf f}) \geq 0$.

We have $\partial^{2}_{00} B = M_{xx}$, $\partial^{2}_{0j} B = \frac{M_{xy}}{\| \nabla f \| \sqrt{R}}  f_{x_{j}}$ for all $j\geq 1$ and 
\begin{align*}
\partial^{2}_{ij} B = \frac{M_{yy}}{\| \nabla f\|^{2} R} f_{x_{i}} f_{x_{j}} - \frac{M_{y}}{\| \nabla f\|^{3} \sqrt{R}} f_{x_{i}} f_{x_{j}} + \frac{M_{y} \delta_{ij}}{\| \nabla f\| \sqrt{R}} \quad \text{for} \quad i,j \geq 1
\end{align*}
where $\delta_{ij}$ is Kronecker symbol. 

Notice that  since $M(x,\sqrt{y})\in C^{2}(\Omega\times \mathbb{R}_{+})$ we have that  $B\in C^{2}(\Omega \times \mathbb{R}^{n})$.  If $\nabla f = 0$ then there is nothing to prove because $W$ becomes diagonal matrix with negative entries on the diagonal. Further assume $\|\nabla f \|\neq 0$. 

Now notice that 
\begin{align*}
W=S\left(W_{1}+\frac{M_{y}\sqrt{R}}{\|\nabla f\|} W_{2}\right)S
\end{align*}
where $S$ is a diagonal matrix with diagonal $(1,\frac{\nabla f}{\| \nabla f\| \sqrt{R}})$, and 
\begin{align*}
W_{1} = 
\begin{bmatrix}
    M_{xx}+\frac{\sqrt{R}\cdot M_{y}}{\|\nabla f\|}       & M_{xy}  & \dots & M_{xy}  \\
    M_{xy}      & M_{yy} & \dots & M_{yy} \\
    \hdotsfor{4} \\
    M_{xy}    &M_{yy} & \dots & M_{yy} 
\end{bmatrix}
\quad \text{and} \quad 
W_{2}=
\begin{bmatrix}
    0      & 0  & 0 & \dots & 0  \\
   0      & \frac{\| \nabla f\|^{2}}{(f_{x_{1}})^{2}} -1& -1& \dots & -1 \\
      0      &  -1&\frac{\| \nabla f\|^{2}}{(f_{x_{2}} )^{2}} -1& \dots & -1 \\
    \hdotsfor{5} \\
    0    &-1 & \dots & -1& \frac{\| \nabla f\|^{2}}{(f_{x_{n}} )^{2}} -1 
\end{bmatrix}
\end{align*}
It is clear that $W_{1}\leq 0$ because $M$ satisfies (\ref{mdiff}) at point $x$ and $\frac{y}{\sqrt{R}}$.

For the  $W_{2}$, first notice that if $f_{x_{j}}\neq 0$ for all $j\geq 1$ then $W_{2}$ is well defined and $W_{2}\geq 0$. Otherwise if $f_{x_{j}}=0$ for some $j$, then consider initial expression $SW_{2}S$  and notice that $SW_{2}S=S\tilde{W}_{2}S+D$, where $\tilde{W}_{2}$ is the same as $W_{2}$ except $j$th column and row are replaced by zeros, and $D$ is zero matrix except the element $(j,j)$ is equal to $\frac{1}{R}$. We again see that $SW_{2}S \geq 0$. Hence $M_{y} SW_{2}S \leq 0$ as soon as (\ref{mdiff}) hods. 
 
 Thus we have proved that if $M(x,\sqrt{y})\in C^{2}(\Omega\times \mathbb{R}_{+}),$ $M$ satisfies (\ref{mdiff})  then by Theorem~\ref{absth} we  have
 \begin{align}\label{interpolate}
 P_{t} M(f,\| \nabla f\|)\leq M(P_{t} f, \| \nabla P_{t} f\|)\quad \text{for all} \quad f\in C^{\infty}(\mathbb{R}^{n};\; \Omega).
 \end{align} 

We send $t\to \infty$ and because of the fact $\| \nabla P_{t} f\|\leq e^{-tR}P_{t} \|\nabla f\|$ (see~\cite{BGL}) we obtain 
\begin{align*}
\int_{\mathbb{R}^{n}}M(f,\|\nabla f\|)d\mu \leq M\left(\int_{\mathbb{R}^{n}} fd\mu, 0 \right),
\end{align*}
where $d\mu=e^{-U(x)}dx$ is a probability measure. 

\begin{remark}
It is worth mentioning but not necessary for our purposes  that (\ref{interpolate}) also implies (\ref{mdiff}) in case of Gaussian measure. This follows from the fact that the matrix $\lambda \Gamma(D^{\alpha}{\bf f})$ can be an arbitrary positive definite matrix where $\lambda >0$ and ${\bf f} \in C^{\infty}(\mathbb{R}^{n};\; \Lambda)$. Then condition $\mathrm{Tr} (W  \Gamma(D^{\alpha}{\bf f}))\leq 0$ implies that $W \leq 0$ and this  gives us condition (\ref{mdiff}). 
\end{remark}
\subsection{Relation to stochastic calculus and $\Gamma$-calculus approach}
Inequality (\ref{interpolate}) implies that the map 
\begin{align}\label{mont}
t \to \int_{\mathbb{R}^{n}}M(P_{t} f, \|\nabla P_{t} f\|) d\gamma
\end{align}
is monotone provided that $M$ satisfies (\ref{mdiff}). Indeed, by sending $t \to 0$ in (\ref{interpolate}) we obtain its infinitesimal form 
$L  M(f, \| \nabla f \|) \leq \left. \frac{d}{ds} M(P_{s} f, \| \nabla P_{s} f\|) \right|_{s =0}.$
Finally, if the last inequality is true for any $f$ then it is true for any $f$ of the form $P_{t} f$. This implies that
\begin{align*}
L  M(P_{t} f, \| \nabla P_{t} f \|) \leq \left. \frac{d}{ds} M(P_{s} f, \| \nabla P_{s} f\|) \right|_{s =t},
\end{align*}
 and it gives monotonicity of (\ref{mont}).  
  
  Interpolation (\ref{interpolate}) (or even monotonicity (\ref{mont})) plays a fundamental role in functional inequalities and it was known before for some particular functions $M(x,y)$ as a consequence of their special properties and some linear algebraic manipulations (see \cite{BGL}, \cite{BaMa}). The purpose of Theorem~\ref{mth} was to exclude the linear algebra involved in the interpolation (\ref{interpolate}) and to show that in fact (\ref{interpolate})  boils down (actually it is equivalent) to the fact that $M$ satisfies  an \emph{ elliptic Monge--Amp\`ere equation of a general form} (\ref{mdiff}). Monge--Amp\`er\`ere equation (\ref{mdiff}),  apparently, was not noticed before or it was hidden in the literature from the wide audience.  Equations of Monge--Amp\`ere type are of course among the most important fully nonlinear partial differential equations (see \cite{PHFI}, \cite{TrWa}).
  
Next we will show that  in fact   (\ref{mdiff}) gives monotonicity of the type (\ref{mont}) in different settings as well.

\medskip

  \subsubsection{Stochastic calculus approach}
\begin{proposition}\label{stochastic}
For $t \geq 0$, let $W_{t}, N_{t}$ be $\mathcal{F}_{t}$ real-valued martingales with $W_{t} = W_{0}+\int_{0}^{t}w_{s} dB_{s}, N_{t} = N_{0}+\int_{0}^{t} n_{s} dB_{s}$, and let $A_{t} = A_{0} + \int_{0}^{t} a_{s} ds$ where $A_{0}, a_{s} \geq 0$. Assume that $A_{t}$ is bounded, $a_{t} |N_{t}|^{2} \geq |w_{t}|^{2}$ and $W_{t} \in \Omega$ for $t \geq 0$. Assume that $M(x,\sqrt{y}) \in C^{2}(\Omega\times \mathbb{R})$ satisfies (\ref{mdiff}). Then 
\begin{align*}
z_{t} = M(W_{t}, |N_{t}| \sqrt{A_{t}})
\end{align*}
is a supermartingale for $t \geq 0$. 
\end{proposition}
\begin{proof}
The proof of the proposition proceeds absolutely  in the same way as in \cite{BaMa}, which treat the case of a {\it particular} function $M$ involved in Bobkov's inequality. In fact, it is (\ref{mdiff})  which makes the drift $\Delta(t)$ nonpositive where $dz_{t} = u_{t} dB_{t} + \Delta(t) dt$. 
\end{proof}
   
 One may obtain another proof of Theorem~\ref{mth} for the case $n=1$ using Proposition~\ref{stochastic} for the special case $W_{t} = \mathbb{E} [f(B_{1})|\mathcal{F}_{t}]$, $A_{t} = t$ and $N_{t} = \mathbb{E}[\nabla f(B_{1})|\mathcal{F}_{t}]$ where $0 \leq t \leq 1$, $B_{0}=0$,  and  $f$ is a real valued smooth bounded function. Indeed, in this case by optional stopping theorem one obtains 
 \begin{align*}
 M(\mathbb{E} [f(B_{1})|\mathcal{F}_{0}],0)=z_{0}\geq \mathbb{E}z_{1} = \mathbb{E} M(f(B_{1}), |\nabla f(B_{1})| ).
 \end{align*}

\subsubsection{$\Gamma$-calculus approach}
One may obtain Theorem~\ref{mth} using the remarkable  $\Gamma$-calculus (see \cite{BGL}). In fact, setting $\Gamma(f,g) = \nabla f \cdot \nabla g$ one can easily show  that  if $M(x,\sqrt{y})\in C^{2}(\Omega\times \mathbb{R}_{+})$ satisfies (\ref{mdiff}) then the following map 
\begin{align*}
s \to P_{s} B(P_{t-s}f, \Gamma(P_{t-s}f, P_{t-s}f))
\end{align*}
is monotone for $0\leq s \leq t$ for any given $t>0$, where $B(x,y^{2}) = M(x,y)$. Initially this was the way we obtained Theorem~\ref{mth} (see for example \cite{paataphd}).  Later it became clear to us that one does not have to be limited with notations $\Gamma, \Gamma_{2}, \Gamma_{3} $ etc in order to enjoy interpolations of the form (\ref{interpolate}).    In fact, one can directly work with an arbitrary differential operator $D^{\alpha}$ (\ref{dop}), and Theorem~\ref{mth} is just a consequence of  Theorem~\ref{absth} for an appropriate choice of \emph{Bellman function} (\ref{bfun1}). In support of classical notations we should say that it is not clear for us how $\Gamma$-calculus can be used in proving Theorem~\ref{absth} which is a simple statement if working with the classical notations of differential operators $D^{\alpha}$.


\section{Reduction to the exterior differential systems and backwards heat equation}
\label{ets}
As we have already mentioned in Section~\ref{conclusions} (and it also follows from the proof of Theorem~\ref{mth}) in order inequality (\ref{gcase}) to be \emph{sharp}  we need to assume that (\ref{mdiff}) degenerates i.e., 
\begin{align}\label{degen}
\det \begin{pmatrix}
M_{xx}+\frac{M_{y}}{y} & M_{xy}\\
M_{xy} & M_{yy}
\end{pmatrix} = M_{xx}M_{yy}-M_{xy}^{2}+\frac{M_{y}M_{yy}}{y}= 0.
\end{align}
Let us make the following observation: consider 1-graph of $M(x,y)$ i.e., 
\begin{align*}
(x,y,p,q)=(x,y,M_{x}(x,y),M_{y}(x,y))
\end{align*} 
in $xypq$-space. This is a simply-connected surface $\Sigma$ in $4$-space on which $\Upsilon=\dd x \wedge \dd y$ is non-vanishing but to which the two 2-forms 
\begin{align*}
\Upsilon_{1}=\dd p \wedge \dd x+\dd q\wedge \dd y \quad \text{and} \quad \Upsilon_{2}=(y\dd p+q \dd x)\wedge \dd q
\end{align*}
pull back to be zero.

Conversely, suppose given simply connected surface $\Sigma$ in $xypq$-space (with $y>0$) on which $\Upsilon$ is non-vanishing but to which $\Upsilon_{1}$ and $\Upsilon_{2}$ pullback to be zero. The 1-form $p\dd x + q\dd y$ pull back to $\Sigma$ to be closed (since $\Upsilon_{1}$ vanishes on $\Sigma$) and hence exact, and therefore there exists a function $m : \Sigma \to \mathbb{R}$ such that $\dd m = p\dd x+q \dd y$ on $\Sigma$. We then have (at least locally), $m=M(x,y)$ on $\Sigma$ and, by its definition, we have $p=M_{x}(x,y)$ and $q=M_{y}(x,y)$ on the surface. Then fact that $\Upsilon_{2}$ vanishes when pulled back to $\Sigma$ implies that $M(x,y)$ satisfies the desired equation. 

Thus, we have encoded the given PDE as an exterior differential system on $\mathbb{R}^{4}$. Note, that we can make a change of variables on the open set where $q<0$: Set $y=qr$ and let $t=\frac{1}{2}q^{2}$. then, using these new coordinates on this domain, we have 
\begin{align*}
\Upsilon_{1}=\dd p \wedge \dd x + \dd t \wedge \dd r \quad \text{and} \quad \Upsilon_{2} = (r\dd p + \dd x)\wedge \dd t.
\end{align*}
Now, when we take an integral surface $\Sigma$ on these $2$-forms on which $\dd p \wedge \dd t$ is vanishing, it can be written locally as a graph of the form 
\begin{align*}
(p,t,x,r)=(p,t,u_{p}(p,t), u_{t}(p,t))
\end{align*}
(since $\Sigma$ is an integral of $\Upsilon_{1}$), where $u(p,t)$ satisfies $u_{t}+u_{pp}=0$ (since $\Sigma$ is an integral of $\Upsilon_{2}$). Thus, ``generically'' our PDE is equivalent to the backwards heat equation, up to a change of variables. Thus the function $M(x,y)$ can be parametrized as follows 
\begin{align}
&x = u_{p}\left(p,\frac{1}{2}q^{2} \right); \quad y=q u_{t}\left(p,\frac{1}{2}q^{2} \right); \nonumber \\
& M(x,y)=pu_{p}\left(p,\frac{1}{2}q^{2} \right) + q^{2}u_{t}\left(p,\frac{1}{2}q^{2} \right) - u\left(p,\frac{1}{2}q^{2}\right). \label{vyrm}
\end{align}
Note that  $y\geq 0,$ $q=M_{y} \leq 0$ then $u_{t}\left(p,\frac{1}{2}q^{2} \right)\leq 0$. 
Let us rewrite the conditions $M_{yy}\leq 0$ and $M_{xx}+\frac{M_{y}}{y}\leq 0$ in terms of $u(p,t)$. In other words we want $q_{y}$ and $p_{x}+\frac{q}{y} \leq 0$. We have 
\begin{align*}
0=u_{pp}p_{y} + u_{pt}qq_{y} \quad \text{and} \quad 1=q_{y} u_{t} + qp_{y} u_{tp}+q^{2}q_{y}u_{tt}.
\end{align*}
Then 
\begin{align*}
1=q_{y}u_{t}+q^{2} q_{y} \frac{u_{pt}^{2}}{u_{pp}}+q^{2}q_{y}u_{tt} \quad \text{and} \quad M_{yy}=q_{y}=\frac{u_{t}}{u_{t}^{2}-2t(u_{tt}u_{pp}-u_{pt}^{2})}.
\end{align*}
Thus the negative definiteness of the matrix (\ref{mdiff}) (if its determinant is known to be zero) is equivalent to 
\begin{align}\label{concavity}
u_{t}^{2}-2t\, \mathrm{det}(\mathrm{Hess}\, u)\geq 0. 
\end{align}

Let us show that the function $u(p,t)$ must satisfy a boundary condition:
\begin{align}\label{boundary}
u(f'(x),0)=xf'(x)-f(x) \quad \text{for} \quad x \in \Omega \quad \text{where} \quad f(x)=M(x,0).
\end{align}
Indeed, we know that $M(x,\sqrt{y})\in C^{2}(\Omega\times \mathbb{R}_{+})$ therefore $M_{y}(x,0)=0$. By choosing $y=0$ in (\ref{vyrm}), we have $q=0$, and we obtain the desired boundary condition: 
\begin{align*}
M(x,0)=xM_{x}(x,0)-u(M_{x}(x,0),0).
\end{align*}
\bigskip
Now it is clear how to find the function $M(x,y)$ provided that $M(x,0)$ is given:  First we try to find a function $u(p,t)$ such that 
\begin{align}
&u_{pp}+u_{t}=0, \quad u_{t} \leq 0, \label{heat} \\
&u(M_{x}(x,0),0)=xM_{x}(x,0)-M(x,0) \quad x \in \Omega, \label{um}\\
&u_{t}^{2}-2t\, \det (\mathrm{Hess}\, u) \geq 0. \label{hess}
\end{align}
Then a candidate  for $M(x,y)$ will be given  by (\ref{vyrm}). We should mention that if $M(x,0)$ is convex then  (\ref{um}) simply means that $u_{p}(p,0)$ is Legendre transform of $M_{x}(x,0)$. Indeed, if we take derivative in (\ref{um}) with respect to $x$ we obtain $u_{p}(M_{x}(x,0),0)=x$. 

\subsection{Back to the applications, old and new. Revisiting Section \ref{appl} with our new tool.}
\label{app12}
\label{goback} Further we assume that we know the expression $M(x,0)$  and we would like to restore the function $M(x,y)$ which satisfies conditions of Theorem~\ref{mth}, PDE~(\ref{degen})  and hence it gives us inequality (\ref{corlog}), or the error term in Jensen's inequality (see Section~\ref{conclusions} for the explanations). 

\subsubsection{Gross function}
In this case we have $M(x,0)=x \ln x$. Condition (\ref{um}) can be rewritten as follows $u(p,0)=e^{p-1}$ for all $p \in \mathbb{R}$. 
If we set $D=\frac{\partial^{2} }{\partial p^{2}}$ then 
\begin{align*}
u(p,t)=e^{-tD} e^{p-1}= \sum_{k=0}^{\infty}\frac{(-t)^{k}}{k!}e^{p-1}=e^{p-t-1} \quad \text{for all} \quad t\geq 0.
\end{align*}
Clearly $u(p,t)$ satisfies (\ref{hess}) because $\det(\mathrm{Hess}\, u)=0$.  Notice that we have $u_{t} < 0$,   
\begin{align*}
\begin{cases}
x=e^{p-\frac{q^{2}}{2}-1};\\
y=-qe^{p-\frac{q^{2}}{2}-1};
\end{cases}
\quad \text{then} \quad 
\begin{cases}
q=-\frac{y}{x};\\
p = \ln x +\frac{y^{2}}{2x^{2}}+1.
\end{cases}
\end{align*}
Therefore we obtain 
\begin{align*}
M(x,y)=xp+qy-u\left( p,\frac{1}{2}q^{2}\right)=x \ln x + \frac{y^{2}}{2x}+x-\frac{y^{2}}{x}-x=x\ln x - \frac{y^{2}}{2x}.
\end{align*}
 
\subsubsection{Nash's function} In this case we have $M(x,0)=x^{2}$.  Condition~(\ref{um}) takes the form $u(p,0)=\frac{p^{2}}{4}$ for all $p \in \mathbb{R}$. Then \begin{align*}
u(p,t)=e^{-tD} \frac{p^{2}}{4}=(1-tD) \frac{p^{2}}{4}=\frac{p^{2}}{4}-\frac{t}{2} \quad t\geq 0. 
\end{align*} 
$u(p,t)$ satisfies (\ref{hess}) because $\det (\mathrm{Hess}\, u)=0$. We have $u_{t} <0$ 
\begin{align*}
\begin{cases}
x=\frac{p}{2};\\
y=-\frac{q}{2};
\end{cases}
\quad \text{then} \quad 
\begin{cases}
p=2x;\\
q=-2y.
\end{cases}
\end{align*}
We obtain 
\begin{align*}
M(x,y)=2x^{2}-2y^{2}-(x^{2}-y^{2})=x^{2}-y^{2}. 
\end{align*}
\subsubsection{Bobkov's function.}
It is not clear at all where the function $M(x,y)=-\sqrt{I(x)^{2}+y^{2}}$ comes from. Apparently it was a pretty good guess.  

Let us show how easily it can be restored by solving Monge--Amp\`ere equation (\ref{Monge}). 
In this case we have $M(x,0)=-I(x)$. 
Condition~(\ref{um}) takes the form 
\begin{align}\label{granica}
u(p,0)=p\Phi(p)+\Phi'(p) \quad \text{for all}\quad   p \in \mathbb{R}. 
\end{align}
Now we will try to find the usual heat extension of $u(p,0)$ (call it $\tilde{u}(p,t)$)  which satisfies $\tilde{u}_{pp}=\tilde{u}_{t}$,  and then we try to consider the formal candidate $u(p,t) := \tilde{u}(p,-t)$.

 It is easier to find the heat extension of $\tilde{u}_{p}(p,0)$ and then take the antiderivative in $p$. Indeed, notice that (\ref{granica}) implies $u_{p}(p,0)=\Phi(p)$. the heat extension of $\Phi(p)$ is $\Phi\left(\frac{p}{\sqrt{1+2t}}\right)$. Indeed, the heat extension of the function $\mathbbm{1}_{(-\infty,0]}(p)$ at time $t=1/2$ is $\Phi(p)$. Then by the semigroup property the heat extension of $\Phi(p)$ at time $t$ will be the heat extension of $\mathbbm{1}_{(-\infty,0]}(p)$ at time $1/2+t$ which equals to $\Phi\left(\frac{p}{\sqrt{1+2t}}\right)$. Thus $\tilde{u}_{p}(p,t)=\Phi\left(\frac{p}{\sqrt{1+2t}}\right)$. Taking antiderivative in $p$ and using (\ref{granica}) if necessary we obtain 
\begin{align*}
\tilde{u}(p,t)=\sqrt{1+2t}\, \Phi'\left(\frac{p}{\sqrt{1+2t}}\right)+p\Phi\left(\frac{p}{\sqrt{1+2t}}\right).
\end{align*}
This expression is well defined even for $t \in (-1/2,0)$. Therefore if we set 
\begin{align*}
u(p,t) = \tilde{u}(p,-t)=\sqrt{1-2t}\, \Phi'\left(\frac{p}{\sqrt{1-2t}}\right)+p\Phi\left(\frac{p}{\sqrt{1-2t}}\right)\quad \text{for}\quad  p \in \mathbb{R}, \quad t\in \left[0,\frac{1}{2}\right), 
\end{align*}
direct computations show that $u(p,t)$ satisfies (\ref{heat}), (\ref{granica}) and (\ref{hess}) because $\det (\mathrm{Hess}\, u)=-\left( \frac{\Phi'(\frac{p}{\sqrt{1-2t}})}{1-2t} \right)^{2}<0$. We have $u_{t} =-\frac{\Phi'(\frac{p}{\sqrt{1-2t}})}{\sqrt{1-2t}}<0$ and $u_{p}=\Phi\left(\frac{p}{\sqrt{1-2t}}\right)$. Therefore,  
\begin{align*}
\begin{cases}
x=\Phi\left(\frac{p}{\sqrt{1-q^{2}}}\right);\\
y=\frac{-q}{\sqrt{1-q^{2}}} \Phi'(\frac{p}{\sqrt{1-q^{2}}});
\end{cases}
\quad \text{then} \quad 
\begin{cases}
\Phi^{-1}(x)=\frac{p}{\sqrt{1-q^{2}}};\\
y=\frac{-q}{\sqrt{1-q^{2}}}\Phi'(\Phi^{-1}(x)).
\end{cases}
\end{align*}
From the last equalities we obtain $M_{y}=q=-\frac{y}{\sqrt{I^{2}(x)+y^{2}}}$ and $M_{x}=p=\frac{I(x)\Phi^{-1}(x)}{\sqrt{I^{2}(x)+y^{2}}}$ where we remind that $I(x)=\Phi'(\Phi^{-1}(x))$. Then it follows that 
\begin{align*}
M(x,y)=-\sqrt{I^{2}(x)+y^{2}}.
\end{align*}

\subsubsection{Function 3/2}
In this case we have $M(x,0)=x^{3/2}$ for $x \geq 0$. It follows from (\ref{um}) that $u(p,0)=\frac{4}{27} p^{3}$ for $p \geq 0$. The solution of the backwards heat is the Hermite polynomial, i.e., we have $u(p,t)=\frac{4}{27}(p^{3}-6tp)$.  $u(p,t)$ satisfies (\ref{hess}) because $\mathrm{Hess}\, u <0$. Since $p \geq 0$ we have $u_{t} \leq 0$. Next we obtain 
\begin{align*}
\begin{cases}
&x = \frac{4}{9}(p^{2}-q^{2});\\
&y= -\frac{8}{9} pq.
\end{cases}
\quad \text{then} \quad 
\begin{cases}
p=\frac{3}{4}\sqrt{2x+2\sqrt{x^{2}+y^{2}}};\\
q=-\frac{3}{4}\sqrt{-2x+2\sqrt{x^{2}+y^{2}}}.
\end{cases}
\end{align*}
Finally 
\begin{align*}
M(x,y)=xp+qy - u \left(p, \frac{1}{2}q^{2}\right)= \frac{1}{\sqrt{2}}(2x - \sqrt{x^{2}+y^{2}})\sqrt{x+\sqrt{x^{2}+y^{2}}}.
\end{align*}

\subsubsection{Function $\arccos(x)$}\label{peculiar} Consider an increasing convex function
\begin{align*}
M(x,0)= x \arccos(-x)+\sqrt{1-x^{2}}.
\end{align*}
It follows from (\ref{um}) that $u(p,0)=-\sin(p)$ for $p \in [0,\pi]$. The solution of the backwards heat (\ref{heat}) becomes $u(p,t)=-e^{t} \sin(p)$. Notice that $u_{t} \leq 0$ for $p \in [0, \pi]$, and 
\begin{align*}
u_{t}^{2}-2t \det(\mathrm{Hess}\, u) = e^{2t}(2t+\sin^{2}(x)) \geq 0.
\end{align*}

Conditions (\ref{vyrm}) can be rewritten as follows 
\begin{align*}
&x = -e^{q^{2}} \cos(p);\\
&y = -qe^{q^{2}/2} \sin(p);\\
&M(x,y)=px+qy+e^{q^{2}/2}\sin(p) = px + qy - \frac{y}{q}, \quad x \in [-1,1], \quad y \geq 0.
\end{align*}
It follows that the negative number $q$ satisfies the equation 
\begin{align}\label{nnumber}
-q \sqrt{e^{q^{2}}-x^{2}}=y,
\end{align}
and $p = \arccos(-xe^{-q^{2}/2})$. Thus we obtain 
\begin{align*}
M(x,y)=x \arccos(-xe^{-q^{2}/2})+ (1-q^{2})\sqrt{e^{q^{2}}-x^{2}}
\end{align*}
where a negative number $q$ is the unique solution of (\ref{nnumber}). Thus we obtain 
\begin{align}
&\int_{\mathbb{R}^{n}}f \arccos(-f e^{-r/2})+(1-r) \sqrt{e^{r} -f^{2}} d\gamma \leq \label{f23} \\
&\left(\int_{\mathbb{R}^{n}} f d\gamma \right) \arccos\left(-\int_{\mathbb{R}^{n}} f \right) + \sqrt{1-\left(\int_{\mathbb{R}^{n}} f d\gamma \right)^{2}}\nonumber
\end{align}
for any smooth bounded $f :\mathbb{R}^{n} \to (-1,1)$ where $r>0$ solves the equation 
\begin{align*}
\| \nabla f\|^{2} = r (e^{r}-f^{2}).
\end{align*}

One can obtain Poincar\'e inequality from (\ref{f23}). Indeed,  take $f_{\varepsilon} = \varepsilon f$ and send $\varepsilon \to 0$. Notice that
\begin{align*}
&r = \varepsilon^{2} \| \nabla f\|^{2} + O(\varepsilon^{2});\\
&M\left( \int_{\mathbb{R}^{n}} f_{\varepsilon} d\gamma, 0 \right) = 1+ \frac{\pi}{2} \varepsilon \int_{\mathbb{R}^{n}} f d\gamma + \frac{1}{2} \left( \int_{\mathbb{R}^{n}} f d\gamma \right)^{2}\varepsilon^{2} + O(\varepsilon^{2});\\
&M(f_{\varepsilon}, \| \nabla f_{\varepsilon}\|) = 1 + \frac{\pi}{2} f \varepsilon + \frac{1}{2}\left( f^{2} - \| \nabla f\|^{2} \right)\varepsilon^{2}+O(\varepsilon^{2}).
\end{align*}  
Substituting these expressions into (\ref{f23}) and sending $\varepsilon \to 0$ we obtain Poincar\'e inequality.


\section{One dimensional case}\label{oned}
Let $n=1$, and set $\alpha=(\alpha_{0},\ldots, \alpha_{m})$ where $\alpha_{0}=0$, $\alpha_{1}=1, \ldots, \alpha_{n}=n$. Take ${\bf f} = \underbrace{(f,\ldots, f)}_{m+1}$, where $f \in C^{\infty}_{0}(\mathbb{R};\; \Omega)$.  Then $D^{\alpha}{\bf f} = (f,f',f'',\ldots, f^{(m)})$. Given a log-concave probability measure $e^{-U(x)}dx$ such that $U''(x) \geq R>0$,  the associated semigroup $P_{t}$ has the generator $L=d^{2}x-U'(x)dx$. Let ${\bf u}=(u_{0},\tilde{{\bf u}})$ where $\tilde{{\bf u}}=(u_{1},\ldots, u_{m})\in \mathbb{R}^{m}$ is arbitrary and $u_{0} \in \Omega$. Let the function $B(u_{0},\ldots, u_{m})\in C^{2}(\Omega\times \mathbb{R}^{m})$. Let $B_{j} := \frac{\partial B}{\partial u_{j}}$ and $B_{ij}:=\frac{\partial^{2} B}{\partial u_{i} \partial u_{j}}$. Set  
\begin{align*}
L_{j}({\bf u},y)=\sum_{k=j+1}^{m} \binom{k}{j} B_{k}({\bf u}) U^{(k-j+1)}(y)\quad \text{for}\quad  \quad j=0,\ldots, m-1
\end{align*}

\begin{remark}
Notice that if $e^{-U(x)}dx=\frac{1}{\sqrt{2\pi}}e^{-x^{2}/2}dx$ then $L_{j}({\bf u},y)=(j+1) B_{j+1}({\bf u}).$
\end{remark}
Further we assume that $B_{mm}\neq 0$. Theorem~\ref{absth} implies the following corollary:
\begin{corollary}\label{erti}
The following conditions are equivalent:
\begin{itemize}
\item[(i)]  For all ${\bf u} \in \Omega \times \mathbb{R}^{m}$ we have
\begin{align*}B_{mm}\leq 0, \quad \tilde{{\bf u}}\{B_{mj}({\bf u}) B_{mi}({\bf u})-B_{mm}({\bf u}) B_{ij}({\bf u})-\delta_{i-j} \frac{B_{mm}}{u_{j+1}} L_{j}({\bf u},y)\}_{i,j=0}^{m-1} \tilde{{\bf u}}^{T}\leq 0.
\end{align*}
\item[(ii)]  For all $f \in C^{\infty}_{0}(\mathbb{R}^{m};\; \Omega)$ and $ t \geq 0$ we have 
\begin{align*}
P_{t} B(f,f',\ldots, f^{(m)}) \leq B(P_{t} f, P_{t} f', \ldots, P_{t} f^{(m)}).
\end{align*}
\end{itemize}
\end{corollary}
\begin{remark}
If we send $t \to \infty$ then (ii) in the corollary implies an inequality
\begin{align*}
\int_{\mathbb{R}}B(f,f',\ldots, f^{(m)}) d\mu(x) \leq B\left( \int_{\mathbb{R}} fd\mu, 0,\ldots, 0\right) \quad\text{for all} \quad f \in C^{\infty}_{0}(\mathbb{R}^{m};\; \Omega).
\end{align*}
\end{remark}
\begin{proof}
It is enough to show that (i) in Corollary~\ref{erti} is the same as (i) in Theorem~\ref{absth}. Notice that 
\begin{align*}
[L,D^{\alpha_{0}}]=0\quad \text{and}\quad  \quad [L,D^{\alpha_{k}}]=\sum_{\ell=1}^{k}\binom{k}{\ell}U^{(\ell+1)}(x)d^{k+1-\ell}x \quad \text{for}\quad 1\leq k \leq m.
\end{align*}
Thus 
\begin{align*}
\nabla B [L,D^{\alpha}]{\bf f}=\sum_{k=1}^{n}\frac{\partial B}{\partial u_{k}}\left(\sum_{\ell=1}^{k}\binom{k}{\ell}U^{(\ell+1)}(x)(P_{t}f)^{k+1-\ell} \right)
\end{align*}
and 
\begin{align*}
\Gamma(D^{\alpha}f)=
\begin{bmatrix}
    g'\cdot g'       & g' \cdot g''  & \dots & g'\cdot g^{(m+1)}  \\
    g''\cdot g'      & g''\cdot g'' & \dots & g''\cdot g^{(m+1)} \\
    \hdotsfor{4} \\
    g^{(m+1)}\cdot g'   &g''\cdot g^{(m+1)} & \dots &g^{(m+1)}\cdot g^{(m+1)}  
\end{bmatrix}
\end{align*}

Therefore quantity (i) in Theorem~\ref{absth} takes the following form 
\begin{align*}
\sum_{k=1}^{m} B_{k}(\s{u})\left[ \sum_{\ell=1}^{k}\binom{k}{\ell} U^{(\ell+1)}(y)u_{k+1-\ell}\right]+\sum_{i,j=0}^{m}B_{ij}(\s{u})u_{i+1}u_{j+1}
\end{align*}
where $u_{1},\ldots, u_{n+1},y$ are arbitrary  real numbers and $u_{0}$ takes values in $\Omega$. Notice that the above expression can be rewritten as follows 
\begin{align*}
B_{mm}u_{m+1}^{2}+2u_{m+1}\left(\sum_{j=0}^{m-1}B_{mj}u_{j+1}\right)+\sum_{i,j=0}^{m-1}B_{ij}u_{i+1}u_{j+1}+ \sum_{k=1}^{m} B_{k}(\s{u})\left[ \sum_{\ell=1}^{k}\binom{k}{\ell} U^{(\ell+1)}(y)u_{k+1-\ell}\right]
\end{align*}

 This expression is nonpositive if and only if  condition (i) of Corollary~\ref{erti} holds. 
\end{proof}
 
 \section{Further applications}\label{further}
 
Houdr\'e--Kagan~\cite{Hodre} obtained an extension of the classical Poincar\'e inequality: 
\begin{align}\label{hoho}
\sum_{k=1}^{2d}\frac{(-1)^{k+1}}{k!}\int_{\mathbb{R}^{n}} \| \nabla^{k} f\|^{2} d\gamma \leq \int_{\mathbb{R}^{n}} f^{2} d\gamma - \left(\int_{\mathbb{R}^{n}} f\right)^{2} \leq \sum_{k=1}^{2d-1}\frac{(-1)^{k+1}}{k!}\int_{\mathbb{R}^{n}}\| \nabla^{k} f\|^{2} d\gamma
\end{align}
for all compactly supported functions $f$ on $\mathbb{R}^{n}$, and any $d\geq 1$.  Here by symbol $\|\nabla^{k} f\|$ we denote 
\begin{align*}
\| \nabla^{k} f\|^{2} =\sum_{|\alpha| = k}  (D^{\alpha} f)^{2}.
\end{align*}
We refer the reader to~\cite{ionel} for further remarks on (\ref{hoho})  in one dimensional case $n=1$. A remarkable paper \cite{ledouxparts} explains (\ref{hoho}) via integration by parts. 

We will illustrate the use of Corollary~\ref{erti} on (\ref{hoho}) in case $n=1$.

\bigskip

{\em Proof of (\ref{hoho}) in case $n=1$.}
 
Consider 
\begin{align*}
B(u_{0},u_{1},\ldots, u_{m})=\sum_{k=0}^{m} \frac{(-1)^{k}}{k!} u_{k}^{2},
\end{align*}
and $d\mu=\frac{1}{\sqrt{2\pi}}e^{-x^{2}/2}dx$. If $m$ is odd then $B_{mm}\leq 0$ and condition (i) of Corollary~\ref{erti} holds. Indeed, in this case 
$L_{j}({\bf u},y)=B_{j+1}({\bf u})(j+1)=u_{j+1} (-1)^{j+1} \frac{2}{j!}$, and 
\begin{align*}
 &\tilde{{\bf u}}\{B_{mj}({\bf u}) B_{mi}({\bf u})-B_{mm}({\bf u}) B_{ij}({\bf u})-\delta_{i-j} \frac{B_{mm}}{u_{j+1}} L_{j}({\bf u},y)\}_{i,j=0}^{m-1} \tilde{{\bf u}}^{T}=\\
 &-B_{mm}\tilde{{\bf u}}\left\{ B_{jj}+B_{j+1}\frac{j+1}{u_{j+1}}\right\}_{i,j=0}^{m-1}\tilde{{\bf u}}^{T}=0.
\end{align*}

Thus by (ii) of Corollary~\ref{erti} we obtain  that for all $f\in C^{\infty}_{0}(\mathbb{R})$
\begin{align*}
\int_{\mathbb{R}} \sum_{k=0}^{m} \frac{(-1)^{k}}{k!} [f^{(k)}(x)]^{2} d\mu \leq \left(\int_{\mathbb{R}} f(x) d\mu \right)^{2}
\end{align*}
for odd $m$, and similarly we obtain the opposite inequality for even $m$. 

\bigskip

{\em Proof of (\ref{cfm1}) (Banaszczyk conjecture)}. We will show that {\em if $f$ is even and $d\mu=e^{-U(x)}dx$ is an even log-concave measure such that $\mathrm{Hess}\, U \geq Id$ then 
\begin{align}\label{coridor}
\left( \int_{\mathbb{R}^{n}} f^{2} d\mu \right)^{2} - \left(\int_{\mathbb{R}^{n}} fd\mu \right)^{2} \leq \frac{1}{2}\int_{\mathbb{R}^{n}} \| \nabla f\|^{2} d\mu.
\end{align}
}
Indeed, take $M(x,y)$ as in (\ref{f5}) i.e., 
\begin{align*}
M(x,y)=x^{2}-\frac{y^{2}}{2} \quad \text{for} \quad x \in \mathbb{R}, \quad y \geq 0. 
\end{align*}
Unfortunately $M(x,y)$ does not satisfy (\ref{mdiff}) (because $M_{xx}+M_{y}/y=1>0$) therefore we cannot directly apply Theorem~\ref{mth}. 

Let $P_{t}$ be the associated semigroup to $d\mu$ and let $L$ be its generator. Consider the function $V(x,t)=P_{t}M(f,\| \nabla f\|)-M(P_{t} f, \|\nabla P_{t} f \|)$ as in the proof of Theorem~\ref{absth}. Then 
\begin{align*}
&(\partial_{t} -L) V(x,t)= -\nabla P_{t} f (\mathrm{Hess}\, U) (\nabla P_{t} f)^{T} + 2\| \nabla P_{t} f\|^{2} - \| \nabla^{2} P_{t} f\|^{2}\leq \\
&\| \nabla P_{t} f (x) \|^{2} - \| \nabla^{2} P_{t} f\|^{2}.
\end{align*}
Clearly it is not true that the above expression is pointwise i.e., for all $x \in \mathbb{R}^{n}$, non positive (consider $t=0$).  Therefore we cannot directly apply maximum principle as in the proof of Theorem~\ref{absth} in order to get pointwise bound  $V(x,t) \leq  0$. Actually we do not need pointwise estimate $V(x,t)\leq  0$ in order to get (\ref{coridor}), for example $\int_{\mathbb{R}^{n}} V(x,t) d\mu \leq  0$ will be enough. Notice that 
\begin{align*}
\int_{\mathbb{R}^{n}} V(x,T) d\mu =\int_{0}^{T} \int_{\mathbb{R}^{n}}  (\partial_{t} -L) V(x,t)d\mu dt   \leq \int_{0}^{T} \int_{\mathbb{R}^{n}} \| \nabla P_{t} f (x) \|^{2} - \| \nabla^{2} P_{t} f\|^{2} d\mu ds\leq 0
\end{align*} 
for all $T\geq 0$. The last inequality follows from the   application of Poincar\'e inequality (\ref{poin}) to the functions $\partial_{x_{j}} P_{t} f(x)$ for all $j=1,\ldots, n$,  and the fact that $ \int_{\mathbb{R}^{n}} \partial_{x_{j}} P_{t} f(x) d\mu=0$ because $P_{t} f(x)$ is even function. Thus we obtain that 
\begin{align*}
\int_{\mathbb{R}^{n}} M(f,\|\nabla f\|) d\mu \leq \int_{\mathbb{R}^{n}}M(P_{T}f, \| \nabla P_{T}f\|)d\mu \quad \text{for all} \quad T\geq 0. 
\end{align*}
By sending $T\to \infty$ we arrive at (\ref{coridor}) because $\lim_{T\to \infty}\|\nabla P_{T} f\| = 0$. 

$\hfill \square$ 

In the end we should mention that even though the current paper is self-contained it should be considered as a continuation of the ideas developed in our recent papers~\cite{pivo, pivo1} where  similar  to (\ref{mdiff}) PDEs happen to rule some functional inequalities.


\begin{thebibliography}{99999}
\bibitem{BaMa} F.~Barthe, B.~Maurey, \emph{Some remarks on isoperimetry of Gaussian type}. Ann. Inst. H. Poincar\'e Probab. Statist., 36 (4) : 419--434, 2000
\bibitem{BE} D.~Bakry, M.~Emery, \emph{Hypercontractivit\'e de semi-groups de diffusion,} C.~R.~Acad.~Sci.~Paris S\'er I Math. {\bf 299}, 775--778 (1984).
\bibitem{BGL} D.~Bakry, I.~Gentil, M.~Ledoux, \emph{Analysis and Geometry of Markov Diffusion Operators}, Grundlehren der Mathematischen Wissenschaften {\bf 348}. Springer, Cham.
\bibitem{Ba-L} D.~Bakry, M.~Ledoux, \emph{L\'evy--Gromov's isoperimetric inequality for an infinity dimensional diffusion generator.} Invent. math. 123, 259--281 (1996) 
\bibitem{beck2} W.~Beckner, \emph{A generalized Poincar\'e inequality for Gaussian measures,} Proceedings of the American Mathematical Society  {\bf 105}, no. 2, 397--400 (1989). 
\bibitem{Bob1} S.~G.~Bobkov, \emph{Isoperimetric and analytic inequalities for log-concave probability measures,} Ann. Probab. {\bf 27}, no. 4, 1903--1921 (1999). 
\bibitem{bob3} S.~G.~Bobkov, \emph{An isoperimetric inequality on the discrete cube, and an elementary proof of the isoperimetric inequality in Gauss space} The Annals of Probability  {\bf 25}, no. 1, 206--214 (1997). 
\bibitem{BoL} S.~G.~Bobkov, M.~Ledoux \emph{From Brunn-Minkowski to Brascamp-Lieb and to logarithmic Sobolev inequalities} Geom. Funct. Anal., 10(5) (2000), 1028--1052.
\bibitem{abo1} A.~Bonami, \emph{Ensembles $\Lambda(p)$ dans le dual de $D^{\infty}$} . Ann. Inst. Fourier, 18 (2) 193--204 (1968).
\bibitem{abo2} A.~Bonami, \emph{\'Etude des coefficients de Fourier des fonctions de $L^{p}(G)$}. Ann. Inst. Fourier, 20:  335--402 (1970).
\bibitem{BL1} H.~J.~Brascamp, E.~H.~Lieb, \emph{On extensions of the Brunn--Minkowski and Pr\'ekopa--Leindler theorems, including inequalities for log-concave functions, and with an application to the diffusion equation,} J.~Funct.~Anal. {\bf 22}, 366--389 (1976).
\bibitem{br1} R.~Bryant et al., \emph{Exterior differential systems}. 
\bibitem{br2} R.~Bryant, P.~Griffiths \emph{Characteristic cohomology of differential systems II: Conservation laws for a class of parabolic equations,} Duke  Math.~J.~ {\bf 78}, no. 3 (1995), 531--676.  
\bibitem{br3} R.~Bryant, P.~Griffiths, \emph{Characteristic cohomology of differential systems I: General theory,} J.~Amer.~Math.~Soc. {\bf 8} (1995), 507--596.
\bibitem{rob3}  Robert Bryant (http://mathoverflow.net/users/13972/robert-bryant), {\em Monge--Amp\`ere with drift}, URL (version 2015-06-29): http://mathoverflow.net/q/210127
\bibitem{caf} L.~Caffarelli, \emph{Monotonicity properties of optimal transportation and the FKG and related inequalities}, Comm. Math. Phys.  {\bf 214}, (2000), 547--563. 
\bibitem{CCE} E.~A.~Carlen, D.~Cordero-Erausquin,  E.~H.~Lieb, \emph{Asymmetric covariance estimates of Brascamp-Lieb type and related inequalities for log-concave measures},  Ann. Inst. Henri Poincar\'e Probab. Stat. {\bf 49} (2013),  1--12.
\bibitem{chafai} D.~Chafai, \emph{On $\Phi$-entropies and $\Phi$-Sobolev inequalities}, preprint 2002.
\bibitem{CFM} D.~Cordero-Erausquin, M.~Fradelizi, B.~Maurey, \emph{The (B) conjecture for the Gaussian measure of dilates of symmetric convex sets and related problems,} J.~Funct. Anal., {\bf 214} (2004), no. 2, 410--427.
\bibitem{Gross} L.~Gross, \emph{Logarithmic Sobolev inequalities.} Amer. J. Math. {\bf 97} (1975), no. 4, 1061--1083.
\bibitem{Hodre} C.~Houdr\'e, A.~Kagan, \emph{Variance inequalities for functions of Gaussian variables}, J.~Theoret.~Probab.  {\bf 8} (1995), no. 1,  23--30
\bibitem{Hu} Y.~Hu, \emph{A unified approach to several inequalities for Gaussian and diffusion measures}. S\'eminare de Probabilit\'es XXXIV. Lecture Notes in Math. 1729, 329--335 (2000). Springer.
\bibitem{paataphd} P.~Ivanisvili, \emph{Geometric aspects of exact solutions of Bellman equations of harmonic analysis problems}, PhD thesis, 2015.
 \bibitem{PIVO5} P.~Ivanisvili, A.~Volberg,  \emph{Improving Beckner's bound via Hermite polynomials}, preprint, arXiv:1606.08500
\bibitem{pivo} P.~Ivanisvili, A.~Volberg, \emph{Bellman partial differential equation and the hill property for classical isoperimetric problems}, preprint (2015)
arXiv: 1506.03409, pp. 1--30.
\bibitem{pivo1} P.~Ivanisvili, A.~Volberg, \emph{Hessian of Bellman functions and uniqueness of the Brascamp--Lieb inequality}, J. London Math. Soc (2015) doi: 10.1112/jlms/jdv040. 
\bibitem{Lat} R.~Lata\l a \emph{On some inequalities for Gaussian measures.} Proceedings of the International Congress of Mathematicians, Vol. II. (Beijing, 2002), 813--822, Higher Ed. Press, Beijing 2002.
\bibitem{Lat22} R.~Lata\l a, K.~Oleszkiewicz, \emph{Between Sobolev and Poincar\'e}, Geometric Aspects of Functional Analysis. Lect. Notes Math., 1745: 147--168, 2000 
\bibitem{ledoux} M.~Ledoux, \emph{Remarks on Gaussian noise stability, Brascamp--Lieb and Slepian inequalities}, Geometric Aspects of Functional Analysis, 309--333, Lecture Notes in Math., 2116, Springer (2014).
\bibitem{ledouxparts} M.~Ledoux, \emph{L'alg\`ebre de Lie des gradients it\'er\'es d'un g\'en\'erateur markovien -- d\'eveloppements de moyennes et entropies}, Annales scientifiques de l'\'E.N.S. $4^{e}$ s\'erie, tome 28, $n^{o}$ 4 (1995), p. 435--460.

\bibitem{Mone} E.~Mossel, J.~Neeman, \emph{Robust optimality of Gaussian noise stability} (2012).
\bibitem{jnash} J.~Nash, \emph{Continuity of solutions of parabolic and elliptic equations}, Amer.~J.~Math. {\bf 88} (1958), 931--954.
\bibitem{RO} R.~O'Donell, \emph{Analysis of Boolean Functions.} (2013).
\bibitem{ionel} I.~Popescu, \emph{A refinement of the Brascamp--Lieb--Poincar\'e inequality in one dimension}, Comptes Rendus Mathematique, {\bf 352}, Issue 1 (2014), 55--58  
\bibitem{PHFI} G.~Philippis, A.~Figalli,  \emph{The Monge--Amp\`ere equation and its link to optimal transportation}, Bulletin of the  AMS, Vo. 5, no. 4, pp. 527-580.
\bibitem{TrWa} N. ~S. ~Trudinger, X.~Wang, \emph{The Monge--Amp\`ere equation and its geometric applications}, Handbook of Geometric Analysis, International Press, 2008, Vol I, pp. 467--524.























   



 
\end{thebibliography}
\end{document}